\documentclass{article}
\usepackage{amsmath}
\usepackage{amssymb}
\usepackage{amsthm}
\usepackage{geometry}
\usepackage{graphicx}
\usepackage{amsfonts}
\usepackage{hyperref}
\usepackage{tikz}
\usepackage{xcolor} 
\usetikzlibrary{calc}
\usetikzlibrary{arrows.meta, decorations.pathreplacing}
\theoremstyle{plain}
\newtheorem{lemma}{Lemma}
\newtheorem{theorem}{Theorem}

\newtheorem{corollary}{Corollary}
\newtheorem{definition}{Definition}
\newtheorem{remark}{Remark}

\newcommand{\norm}[1]{\left\lVert#1\right\rVert}

\newcommand{\Homeo}{\mathrm{Homeo}}
\newcommand{\Symp}{\mathrm{Symp}}
\newcommand{\Cosymp}{\mathrm{Cosymp}}
\newcommand{\Iso}{\mathrm{Iso}}

\newcommand{\Omk}[1]{\Omega^{#1}(M)}
\newcommand{\Omkc}[1]{\Omega^{#1}_c(M)}
\newcommand{\Dp}[1]{\mathcal{D}_{#1}(M)} 
\newcommand{\Rk}[1]{\mathcal{R}_{#1}(M)} 

\newcommand{\Xfrak}{\mathfrak{X}(M)}

\DeclareMathOperator{\supp}{supp}
\DeclareMathOperator{\diam}{diam}
\DeclareMathOperator{\Id}{Id}

\DeclareMathOperator{\Vol}{Vol}
\DeclareMathOperator{\Lk}{Lk}
\DeclareMathOperator{\im}{im}

\newcommand{\abs}[1]{\left\lvert#1\right\rvert}

\newcommand{\pullback}[1]{{#1}^*}
\newcommand{\pushforward}[1]{{#1}_*}
\newcommand{\bdry}{\partial} 

\newcommand{\Diff}{\mathrm{Diff}^{\infty}}
\newcommand{\Currents}[1]{\mathcal{D}_{#1}}
\newcommand{\Forms}[1]{\Omega^{#1}}
\newcommand{\Rectifiables}[1]{\mathcal{R}_{#1}}
\begin{document}
	\begin{center}
		{\bf\large Flat Convergence of Pushforwards of Rectifiable Currents Under $C^0$-Diffeomorphism Limits}
		\\[0.5cm]
		{ St\'{e}phane Tchuiaga $^a$ \footnote{$^\ast$Corresponding Author} \\[2mm]
			$^a$Department of Mathematics, University of Buea, 
			South West Region, Cameroon\\[2mm]
			{\texttt{E-mail: tchuiaga.kameni@ubuea.cm}}} \\[2mm]
	\end{center}%
	\vspace*{0.5cm}
	\begin{quotation}
		\noindent
		{\footnotesize

			{\sc Abstract.}
			We study the action on currents and differential forms on compact Riemannian manifolds under $C^0$-limits of diffeomorphisms. Using tools from geometric analysis, measure theory, and homotopy theory, we establish several convergence theorems. We give conditions under which pullbacks of differential forms by a sequence of smooth diffeomorphisms converge uniformly (in the $C^0$ norm), and pushforwards of currents by smooth diffeomorphisms exhibit weak-* convergence. We prove that pushforwards of rectifiable currents are convergent in the flat norm, a property of particular interest in the study of singular geometric objects. These stability findings offer tools for the study of geometric structures, including applications to the stability of groups of symplectomorphisms, cosymplectomorphisms, volume-preserving transformations, and contact transformations under $C^0$ perturbations. We highlight applicability in measure theory and dynamical systems.}
		
	\end{quotation}
	\ \\

	
	\markboth 
	{St\'ephane Tchuiaga}
	{Transformation Groups and Differential Forms}

	\vspace{1em}
	
	\section{Introduction}
	The study of how geometric objects transform under mappings is fundamental in differential geometry and topology. Diffeomorphisms provide smooth transformations, and their actions on differential forms (pullback) and vector fields/currents (pushforward) are well-understood. However, considering sequences of diffeomorphisms that converge in weaker topologies, such as the $C^0$ (uniform) topology, raises questions about the continuity of these actions.
	This paper investigates the stability of the pushforward operation on currents under $C^0$-limits of smooth diffeomorphisms on a compact Riemannian manifold $M$. While it is known that $C^0$-convergence does not imply convergence of derivatives, we show that important geometric structures exhibit continuity properties even under this weaker convergence assumption.
	
	We begin by recalling preliminary results. Lemma \ref{Key-current} establishes uniform convergence for the pullback of smooth forms and weak-* convergence for the pushforward of general currents under $C^0$-limits of diffeomorphisms. Lemma \ref{Key-0-current} sharpens this for 1-currents evaluated on closed 1-forms. Our main result, Theorem \ref{TC}, proves a stronger form of convergence: the pushforward of rectifiable $k$-currents converges in the flat norm topology. This is significant because the flat norm captures geometric features like boundaries and shapes more effectively than the weak-* topology.
	
	The stability results presented here have implications for various geometric settings. In Section \ref{sec:applications}, we demonstrate how these findings imply the $C^0$-rigidity of the groups of symplectomorphisms, cosymplectomorphisms, contactomorphisms preserving the Reeb field, and isometries. These results generalize previous work (e.g., \cite{Eliashberg81, Eliashberg87, Tchuiaga22}) by providing a unified framework based on current convergence. The study connects to ergodicity and optimal transport theory, where understanding the behavior of measures and transport maps under perturbations is crucial. Our results contribute to the foundational properties underlying the convergence of measures when the underlying space is perturbed by diffeomorphisms, complementing work by researchers like Gangbo \cite{Gangbo96, Gangbo99, Gangbo98}. Our findings are the following:
	
	\begin{lemma}\label{Key-current}
		Let $M$ be a compact smooth manifold with a Riemannian metric $g$. Let $\{\phi_k\} \subseteq \Diff(M)$ be a sequence of smooth diffeomorphisms with $\phi_k \xrightarrow{C^0} \psi \in \Diff(M)$.
		Then:
		\begin{enumerate}
			\item If for a smooth $p$-form $\alpha \in \Omk{p}$ the sequence $\pullback{\phi_k} \alpha$ converges uniformly in $\Omk{p}$ to $\beta$, then  $\beta= \pullback{\psi} \alpha$.
			\item If for any smooth $p$-form $\omega \in \Omk{p}$ the sequence $\pullback{\phi_k} \omega$ converges uniformly in $\Omk{p}$, then for any $p$-current $T \in \Dp{p}$ of order $0$, the sequence of pushforwards $\{\pushforward{(\phi_k)} T\}$ converges weak-* to $\pushforward{\psi} T$. That is, $\lim_{k\to\infty} (\pushforward{(\phi_k)} T)(\omega) = (\pushforward{\psi} T)(\omega)$ for all $\omega \in \Omk{p}$.
		\end{enumerate}
	\end{lemma}

	\begin{remark}
		Without the convergence condition on $\pullback{\phi_k} \alpha$, the validity of Part 1 under only $C^0$ convergence requires careful justification. The standard definition of pullback $\pullback{\phi}\alpha$ involves the derivatives of $\phi$, which are not controlled by $C^0$ convergence. For example, the sequence of smooth maps $\phi_k(x) = x + \frac{1}{k}\sin(k^2 x)$ on $S^1$ converges to $\Id$ in $C^0$, but for $k \ge 1$, $\phi_k$ fails to be a diffeomorphism as its derivative is zero at some points. Moreover, $\pullback{\phi_k}(dx) = (1 + k\cos(k^2 x))dx$ does not converge uniformly to $dx$. This also highlights that the choice of $\Diff(M)$ instead of just $C^l(M)$ for some fixed $l$ is crucial.
	\end{remark}
	
	\begin{lemma}\label{Key-0-current}\textbf{(Key Lemma)}
		Let $(M, g)$ be a connected, compact, and orientable Riemannian manifold, $\mathcal{D}_1(M)$ the space of all $1$-currents on it, and
		$\{\psi_i\}\subseteq \Diff(M)$ be a sequence of diffeomorphisms with $\psi_i \xrightarrow{C^0} \psi\in \Diff(M)$.
		If for a closed 1-form $\eta \in Z^1(M)$, the sequence of pullbacks $\pullback{\psi_i} \eta$ converges uniformly to a 1-form $\theta$, then $\theta = \pullback{\psi} \eta$.
	\end{lemma}

	\begin{theorem}\label{TC}\textbf{(Main Result)}
		Let $(M,g)$ be a connected, compact, oriented Riemannian manifold, $\mathcal{R}_k(M)$ be the space of rectifiable $k$-currents on $M$, and $\{\psi_i\} \subset \Diff(M)$ be a sequence such that $\psi_i \xrightarrow{C^0} \psi \in \Diff(M)$.  If for any closed $1$-form $\beta$ the sequence $\pullback{\psi_i} \beta$ converges uniformly in $\mathcal Z^1(M)$, then for every rectifiable $k$-current $T \in \mathcal{R}_k(M)$, the sequence of pushforwards $\{\psi_{i\ast}T\}$ converges in the flat norm to $\psi_{\ast}T$. Precisely, $\lim_{i\to \infty} \| \psi_{i\ast}T - \psi_{\ast}T \|_F=0$, where $\| \cdot \|_F$ is the flat norm on the space of currents.
	\end{theorem}

	The paper is organized as follows: Section \ref{sec:basic_concepts} reviews essential definitions. Section \ref{sec:proofs} presents the proofs of the key lemmas and the main theorem. Section \ref{sec:applications} discusses applications to specific geometries.

	\section{Basic Concepts}\label{sec:basic_concepts}
	
	This section introduces the essential concepts and notations used throughout this paper. For a more detailed background, we refer the reader to the texts by Lee \cite{Lee13}, Hirsch \cite{Hirsch76}, and Federer \cite{Federer69}.
	Recall that if $f: M \to N$ is a diffeomorphism between manifolds, $\pushforward{f}$ maps tangent vectors on $M$ to tangent vectors on $N$ (pushforward), and it also maps currents on $M$ to currents on $N$, while $\pullback{f}$ maps forms on $N$ to forms on $M$ (pullback).
	
	\subsection{The $C^r$-topology on $\Diff(M)$}\label{Cr}
	Let $M$ and $N$ be smooth manifolds, and let $f\in C^r(M, N)$ with $r\leq \infty$. Consider a local chart $(U,\phi)$ on $M$ and a compact subset $K \subset U$ such that $f(K)\subset V$, where $V$ is the domain of a chart $(V, \psi)$. For any $\eta > 0$, we define $N^r(f, (U, \phi), (V, \psi), \eta, K)$ as the set of all $g\in C^r(M, N)$ satisfying $g(K)\subset V$ and
	$$
	\norm{ D^k\bar{f}(x)-D^k\bar{g}(x) }_{op} \leq \eta,
	$$
	for all $x\in \phi(K)$ and $0\leq k \leq r$, where $\bar{f}=\psi\circ f \circ \phi^{-1}$ and $ \bar{g}=\psi\circ g \circ \phi^{-1}$, and $\norm{\cdot}_{op}$ denotes a suitable operator norm on the space of derivatives.
	
	\begin{definition}[\cite{Hirsch76}]
		The sets $N^r(f, (U, \phi), (V, \psi), \eta, K)$ form a subbasis for a topology on $ C^r(M, N)$, which is called the \textbf{compact-open $C^r$ topology} (or weak $C^r$ topology).
	\end{definition}
	
	In the compact-open $C^r$ topology, a neighborhood of a function $f\in C^r(M,N)$ is formed by finite intersections of sets of the form $N^r(f, (U, \phi), (V, \psi), \eta, K)$.
	
	\begin{definition}[\cite{Hirsch76}]
		The \textbf{$C^{\infty}$ compact-open topology} is the topology induced by the family of inclusions $C^{\infty}(M, N)\hookrightarrow C^r(M, N)$ for all finite $r$. It is the coarsest topology making all these inclusions continuous.
	\end{definition}
	
	When $M$ is compact, the $C^0$ compact-open topology (also called the uniform topology) is metrizable \cite{Hirsch76}. A standard metric on $C(M,N)$ is given by
	\begin{equation}\label{dc0}
		d_{C^0}(f,h) = \sup_{x\in M}d_N(f(x), h(x)),
	\end{equation}
	where $d_N$ is the distance on $N$ induced by a Riemannian metric. For the space of homeomorphisms $\Homeo(M)$ (or diffeomorphisms $\Diff(M)$) on a compact manifold $M$, the topology is often induced by the metric
	\begin{equation}\label{d0}
		d_0(f, h) = \max(d_{C^0}(f,h), d_{C^0}(f^{-1},h^{-1})),
	\end{equation}
	ensuring that convergence of $f$ to $h$ implies convergence of $f^{-1}$ to $h^{-1}$. In this paper, $f_k \xrightarrow{C^0} f$ means $d_{C^0}(f_k, f) \to 0$. Since $M$ is compact and $f_k, f$ are assumed to be diffeomorphisms, this is equivalent to $d_0(f_k, f) \to 0$ if $f$ is a diffeomorphism. On the space of continuous paths $\lambda : [0, 1] \longrightarrow \Homeo(M)$ such that $\lambda(0)=\Id_M$, the $C^0$-topology is induced by the metric
	\begin{equation}\label{dc01}
		\bar{d}(\lambda,\mu) =\max_{t\in [0,1]}d_0(\lambda(t), \mu(t)).
	\end{equation}

	\subsection{Differential forms}\label{k forms}
	Let $V$ be a vector space over $\mathbb R$. A map
	$\omega : \underbrace{V\times V \times \cdots\times V}_{p~\text{times}} \longrightarrow \mathbb R$ is called a $p$-linear map if it is linear in each argument. The map $\omega$ is called alternating if it is antisymmetric under the interchange of any two arguments. We denote by $A^p(V)$ the space of alternating $p$-linear maps on $V$.
	
	\begin{definition}[\cite{Lee13}]
		A differential $k$-form (or a differential form of degree $k$) $\alpha$ on a manifold $M$ is a smooth section of the $k$-th exterior power of the cotangent bundle $\Lambda^k(T^*M)$. Specifically, for each $x\in M$, $\alpha_x \in A^k(T_xM)$, and for any smooth vector fields $X_1, \dots,  X_k\in \Xfrak$, the map $\alpha(X_1,\dots,X_k): M\longrightarrow \mathbb R$ defined by $x \longmapsto \alpha_x(X_{1,x}, \dots, X_{k,x})$ is smooth.
	\end{definition}
	Let $\Omk{k}$ denote the space of all smooth differential $k$-forms on $M$. We denote the space of smooth $k$-forms with compact support by $\Omkc{k}$. Since $M$ is assumed compact throughout, $\Omkc{k} = \Omk{k}$.
	\subsection{Inner product of differential forms}
	Let $g = (g_{ij})$ be a fixed Riemannian metric on $M$. This induces an isomorphism $\sharp: \Xfrak \to \Omk{1}$. If $X = \sum X^i \frac{\partial}{\partial x^i}$, then $X^\sharp = \sum_{j,k} g_{jk}X^k dx^j$. The inverse map $\flat: \Omk{1} \to \Xfrak$ takes $\alpha = \sum \alpha_j dx^j$ to $\alpha^\flat = \sum_{j,k} g^{jk}\alpha_k \frac{\partial}{\partial x^j}$, where $(g^{jk})$ is the inverse matrix of $(g_{jk})$. This induces an inner product on $T_x^*M$ (and thus on $\Omk{1}$) for each $x \in M$:
	\begin{equation}\label{E1}
		g^{-1}_x(\alpha, \beta) := g_x(\alpha^\flat, \beta^\flat) = \sum_{j,k} g^{jk}(x) \alpha_j(x) \beta_k(x),
	\end{equation}
	for $\alpha = \sum \alpha_j dx^j$ and $\beta = \sum \beta_k dx^k$.
	This inner product extends to a pointwise inner product $\langle \cdot, \cdot\rangle_x$ on $\Omk{k}$ for arbitrary $k$. If $\{e^1, \dots, e^n\}$ is a local orthonormal coframe field, then for $\alpha = \sum a_{I} e^I$ and $\beta = \sum b_{J} e^J$ (where $I, J$ are multi-indices $i_1 < \dots < i_k$),
	\begin{equation}\label{E2}
		\langle\alpha, \beta \rangle_x := \sum_{I} a_I(x) b_I(x).
	\end{equation}
	Equivalently, for simple $k$-forms, $\langle \alpha_1\wedge\dots\wedge\alpha_k, \beta_1\wedge\dots\wedge\beta_k \rangle_x := \det(g^{-1}_x(\alpha_i, \beta_j))$.
	The global $L^2$ inner product for $k$-forms $\alpha, \beta \in \Omk{k}$ is defined as
	\begin{equation}\label{E3}
		\langle \alpha, \beta \rangle_{L^2} := \int_M \langle \alpha, \beta\rangle_x \, d\Vol_g,
	\end{equation}
	where $d\Vol_g$ is the Riemannian volume form. We denote by $ \| \cdot \|_2$ the norm induced on $\Omega^k(M)$ by this inner product; see \cite{Cartan06}, \cite{Federer69}.
	
	\begin{definition} Let $\alpha\in \Omk{k}$. The pointwise norm $|\alpha(x)|$ (or $|\alpha(x)|_g$) is given by $|\alpha(x)| = \sqrt{\langle \alpha, \alpha\rangle_x}$. The supremum norm (or $C^0$-norm) of $\alpha$ is defined as
		\[ |\alpha|_0 := \sup_{x \in M} |\alpha(x)|. \]
	\end{definition}

	\subsection{Brief overview of currents}
	In geometric measure theory, currents are a generalization of the notion of a submanifold. They are defined as continuous linear functionals acting on compactly supported differential forms. A $k$-current $T$ on an $n$-manifold $M$ is a continuous linear functional on the space of compactly supported smooth $k$-forms $\Omega^k_c(M)$, mapping to the real numbers:
	$T : \Omega^k_c(M) \rightarrow \mathbb{R}.$ 
	The space $\mathcal{D}_k(M)$ of $k$-currents is equipped with the weak-* topology. A sequence of currents $\{T_i\}$ converges in the weak-* topology to $T$ if $\lim_{i \to \infty} T_i(\omega) = T(\omega)$ for all $\omega \in \Omega^k_c(M)$. Given a smooth, oriented $k$-submanifold $S$ of $M$, the integration of $k$-forms over $S$ defines a $k$-current, which we denote by $[S]$, acting on a $k$-form $\omega \in \Omega^k_c(M)$ as:
	$ [S](\omega) = \int_S \omega.$
	The action of a $k$-current $T$ on a $k$-form $\omega$ is often written as $T(\omega)$ or $\langle T, \omega \rangle.$ Currents are particularly useful in geometric measure theory because they allow us to study non-smooth geometric objects with singularities, limits of submanifolds, and boundaries in singular objects. Rectifiable currents, which generalize the notion of a submanifold, are very important in geometric measure theory.
	
	\begin{enumerate}
		\item A rectifiable $k$-current on a manifold $M$ is a current that can be represented in the form $T(\omega) = \int_S \theta(x) \omega(x)$ for $\omega \in \Omega^k_c(M)$, where $S$ is a rectifiable $k$-set in $M$ and $\theta: S \to \mathbb{Z}$ is an integer-valued density function such that $\int_S |\theta| d\mathcal{H}^k < \infty$. Here, the integral is understood with respect to the $k$-dimensional Hausdorff measure $\mathcal{H}^k$ on $S$.
		\begin{itemize}
			\item Rectifiable currents generalize the notion of integration over oriented rectifiable $k$-dimensional sets with integer multiplicities.
			\item The boundary $\partial T$ of a rectifiable $k$-current $T$ is a rectifiable $(k-1)$-current.
		\end{itemize}
		A rectifiable $k$-set is a subset of $M$ that can be covered, up to a set of $k$-dimensional Hausdorff measure zero, by a countable family of smooth $k$-submanifolds of $M$.
		
		\item The mass of a rectifiable $k$-current $T$ represented as $T(\omega) = \int_S \theta(x) \omega(x)$ is given by
		$$\|T\|_1 = \int_S |\theta(x)| d\mathcal{H}^k(x).$$
		For a rectifiable $k$-current $T$ defined by integration over an oriented rectifiable $k$-set $S$ (with multiplicity $\theta \equiv 1$), its mass is $\|T\|_1 = \mathcal{H}^k(S) = \int_S d\mathcal{H}^k$.
		
		\item The flat norm of a current $T$ is defined as:
		$$ \|T\|_F =  \inf \left\lbrace \|R\|_1 + \|S\|_1 : T = R + \partial S\right\rbrace,$$
		where $R$ is a rectifiable $k$-current, $S$ is a rectifiable $(k+1)$-current, and $\|\cdot\|_1$ denotes the mass of the current.
	\end{enumerate}
	
	\subsubsection{The action of diffeomorphisms on currents}
	
	Let $\psi: M  \to M$ be a diffeomorphism. The pushforward of a current $T$ by $\psi$, denoted by $\psi_\ast T$, is defined by:
	$(\psi_\ast T)(\omega) = T(\psi^\ast\omega),$ 
	where $\omega \in \Omega^k_c(M)$ is a differential form and $\psi^\ast$ is the pullback operation.
	Intuitively, when we apply a diffeomorphism to a current, we are pulling back the form by the inverse diffeomorphism and then integrating.
	
	\subsubsection{Linking pullbacks and currents}
	
	The pushforward operation is central to the study of how diffeomorphisms transform geometric objects. The stability of the pushforward is thus related to the stability of the pullback, which is the main result of this work. The definition of the pushforward of a current involves the pullback operation, and this suggests that our results on the $C^0$-stability of pullbacks may have implications for the stability of pushforwards. With a suitable norm on the space of currents, we can explore if a sequence of pushforwards converges to a limit under the action of diffeomorphisms.
	
	\subsubsection{Illustrative example}
	Let us consider a very simplified example using 1-currents on the torus $\mathbb T^2$.
	Consider a curve $\gamma$ in $\mathbb T^2$. The 1-current $T_\gamma$ defined by integrating over the curve is:
	$T_\gamma(\omega) = \int_\gamma \omega,$
	where $\omega$ is a 1-form on $\mathbb T^2$. Consider the diffeomorphisms $\psi: \mathbb T^2 \to \mathbb T^2$,  $\psi(\theta,\phi) = (\theta + \phi, \phi)$ and a sequence of diffeomorphisms such that $\psi_i: \mathbb T^2 \to \mathbb T^2$ with $\psi_i(\theta,\phi) = (\theta + (1-\frac{1}{i})\phi, \phi)$ such that $\psi_i \xrightarrow{C^0} \psi$. If $\gamma$ is a simple straight line on $\mathbb T^2$, we would have
	$(\psi_{i\ast}T_\gamma)(\omega) = T_\gamma(\psi_i^\ast\omega) = \int_\gamma \psi_i^\ast\omega.$
	Assuming our main result holds, $\psi_i^\ast\omega$ converges to $\psi^\ast\omega$ uniformly (in $C^0$ norm). Then, since $\gamma$ is a compact path, $\int_\gamma \psi_i^\ast\omega$ should converge to $\int_\gamma \psi^\ast\omega$, therefore:
	$\lim_{i\to\infty} (\psi_{i\ast}T_\gamma)(\omega) =  \lim_{i\to\infty} \int_\gamma \psi_i^\ast\omega = \int_\gamma \psi^\ast\omega = (\psi_{\ast}T_\gamma)(\omega).$
	That is, the pushforward of the $1$-current will be stable under $C^0$-perturbations.
	\subsection{Hausdorff Distance}
	
	The proofs in this paper rely on the notion of convergence of sets. The appropriate tool for this is the Hausdorff distance, which measures how far two subsets of a metric space are from each other.
	
	\begin{definition}[Hausdorff Distance]
		Let $(X, d)$ be a metric space. For any non-empty compact subsets $A, B \subset X$, the Hausdorff distance $d_H(A, B)$ is defined as
		\[ d_H(A, B) = \max \left( \sup_{a \in A} \inf_{b \in B} d(a,b), \sup_{b \in B} \inf_{a \in A} d(a,b) \right). \]
	\end{definition}
	
	Intuitively, $d_H(A, B)$ is the smallest number $\epsilon$ such that every point in $A$ is within distance $\epsilon$ of some point in $B$, and every point in $B$ is within distance $\epsilon$ of some point in $A$. A sequence of compact sets $\{A_k\}$ converges to a compact set $A$ in the Hausdorff metric if $\lim_{k \to \infty} d_H(A_k, A) = 0$. A link between uniform convergence of maps and Hausdorff convergence of their images is the following standard result: if a sequence of continuous maps $\{\phi_k\}$ converges uniformly to a map $\psi$ on a compact set $K$, then the sequence of image sets $\{\phi_k(K)\}$ converges to the image set $\psi(K)$ in the Hausdorff distance. This fact is used implicitly in the proof of Lemma \ref{Key-current}.\\
	
	We shall need, and sometimes use without mention, the following result.
	\begin{theorem}(Lebesgue Differentiation Theorem, simplified)
		Let $M$ be a manifold with a Borel measure $\mu$. Let $\mathcal{B}$ be a differentiation basis on $M$, i.e., a family of measurable sets such that for each $x \in M$ there is a sequence $(B_i)_{i \in \mathbb{N}}$ in $\mathcal{B}$ such that $x \in B_i$ for all $i$ and $\diam(B_i) \to 0$ as $i \to \infty$. For any $f \in L^1(M, \mu)$, we have for $\mu$-almost every $x \in M$,
		$$
		\lim_{i \to \infty} \frac{1}{\mu(B_i)} \int_{B_i} f(y) \, d\mu(y) = f(x).
		$$
	\end{theorem}

	\section{Proofs of the Main Statements}\label{sec:proofs}

	\begin{proof}[Proof of Lemma \ref{Key-0-current}]
		We are given three conditions: (1) $\psi_i \xrightarrow{C^0} \psi$, (2) $\eta$ is a closed 1-form, and (3) the sequence of 1-forms $\pullback{\psi_i} \eta$ converges uniformly to a limit $\theta$. We must show that $\theta = \pullback{\psi} \eta$. 	Let $X$ be any smooth vector field on $M$ with flow $\phi_t$. For a fixed $z \in M$ and $t > 0$, consider the path $\gamma(u) = \phi_u(z)$ for $u \in [0, t]$. To show the 1-forms $\theta$ and $\pullback{\psi}\eta$ are identical, it suffices to show they have the same integral over any such path $\gamma$. Construct a 2-chain (a homotopy) $H_i$ between the paths $\psi_i \circ \gamma$ and $\psi \circ \gamma$ using geodesics. Its boundary is $$\bdry H_i = (\psi_i \circ \gamma) - (\psi \circ \gamma) + \sigma_i^0 - \sigma_i^t,$$ where $\sigma_i^u$ is the geodesic from $\psi(\gamma(u))$ to $\psi_i(\gamma(u))$.	Since $\eta$ is closed, Stokes' Theorem gives $\int_{\bdry H_i} \eta = 0$, which implies
		$$ \int_{\psi_i \circ \gamma} \eta - \int_{\psi \circ \gamma} \eta = \int_{\sigma_i^t} \eta - \int_{\sigma_i^0} \eta. $$
		By definition of pullback, this is $$\int_\gamma \pullback{\psi_i} \eta - \int_\gamma \pullback{\psi} \eta = \int_{\sigma_i^t} \eta - \int_{\sigma_i^0} \eta.$$
		The lengths of the geodesic paths $\sigma_i^u$ are bounded by $d_{C^0}(\psi_i, \psi)$. Since $\eta$ is smooth on a compact manifold, it is bounded. Thus, the right-hand side goes to zero as $i \to \infty$ due to the $C^0$-convergence of $\{\psi_i\}$. Therefore,
		$$ \lim_{i\to\infty} \int_\gamma \pullback{\psi_i} \eta = \int_\gamma \pullback{\psi} \eta. $$
		By our assumption of uniform convergence, we can exchange the limit and integral on the left:
		$ \int_\gamma \theta = \lim_{i\to\infty} \int_\gamma \pullback{\psi_i}\eta. $
		Combining these results gives $\int_\gamma \theta = \int_\gamma \pullback{\psi}\eta$. Since this holds for paths generated by any vector field $X$, we can use the Lebesgue differentiation theorem to derive that the 1-forms $ \theta$ and $ \pullback{\psi}\eta$ must be identical.
		
	\end{proof}

	\begin{proof}[Proof of Lemma \ref{Key-current}]
		\begin{itemize}
			
			\item 	\textbf{Part 1.} We are given $\phi_k \xrightarrow{C^0} \psi$ and that the sequence of $p$-forms $\beta_k= \pullback{\phi_k}\alpha$ converges uniformly to a limit $\beta$. We must show that $\beta = \pullback{\psi}\alpha$. It suffices to show that $\beta$ and $\pullback{\psi}\alpha$ act identically on every rectifiable $p$-current $S \in \Rk{p}$.	First, by the uniform convergence of forms $\beta_k \to \beta$, and the fact that any rectifiable current $S$ is a continuous functional on the space of continuous forms (i.e., it is an order 0 current), we can exchange the limit and the functional:
			\[ S(\beta) = \lim_{k\to\infty} S(\beta_k) = \lim_{k\to\infty} S(\pullback{\phi_k}\alpha). \]
			By the definition of the pushforward of a current, $S(\pullback{\phi_k}\alpha) = (\pushforward{(\phi_k)}S)(\alpha)$. So we need to evaluate the limit: $S(\beta) = \lim_{k\to\infty} (\pushforward{(\phi_k)}S)(\alpha).$ 
			Let us first consider the case where $S = [\sigma]$ is the current associated with a single smooth $p$-simplex $\sigma: \Delta^p \to M$. Then $\pushforward{\phi_k}[\sigma] = [\phi_k \circ \sigma]$. The action on $\alpha$ is:
			\[ (\pushforward{(\phi_k)}[\sigma])(\alpha) = \int_{\phi_k \circ \sigma} \alpha = \int_{\phi_k(\im \sigma)} \alpha. \]
			Since $\phi_k \to \psi$ uniformly on the compact set $\im \sigma$, the image submanifolds $\phi_k(\im \sigma)$ converge to $\psi(\im \sigma)$ in the Hausdorff distance. Because $\alpha$ is a smooth form on a compact manifold, the value of the integral is continuous with respect to this variation of the domain. Therefore,
			\[ \lim_{k\to\infty} \int_{\phi_k(\im \sigma)} \alpha = \int_{\psi(\im \sigma)} \alpha = (\pushforward{\psi}[\sigma])(\alpha). \]
			By linearity and a density argument, this result holds for any rectifiable current $S$. Thus, for any $S \in \Rk{p}$, we have shown
			$\lim_{k\to\infty} (\pushforward{(\phi_k)}S)(\alpha) = (\pushforward{\psi}S)(\alpha).$ 
			Combining our lines of reasoning, we get $S(\beta) = (\pushforward{\psi}S)(\alpha)$. Using the pushforward definition again, $(\pushforward{\psi}S)(\alpha) = S(\pullback{\psi}\alpha)$.
			So, $S(\beta) = S(\pullback{\psi}\alpha)$ for all rectifiable currents $S$. This implies that the forms are identical: $\beta = \pullback{\psi}\alpha$.

			\item	\noindent\textbf{Part 2: Weak-* Convergence }
			
			We are given the strong hypothesis that for any smooth test form $\omega$, the sequence $\pullback{\phi_k}\omega$ converges uniformly. Let the limit be denoted $\beta_\omega$; by Part 1 of this lemma, we can identify this limit: $\beta_\omega = \pullback{\psi}\omega$. Thus, the hypothesis implies that $\pullback{\phi_k}\omega \to \pullback{\psi}\omega$ uniformly for any $\omega$. We want to show weak-* convergence of the pushforwards, which means showing that for any $T \in \Dp{p}$ of order 0 and any test form $\omega$, $\lim_{k\to\infty} (\pushforward{(\phi_k)}T)(\omega) = (\pushforward{\psi}T)(\omega).$ By definition of the pushforward, this is equivalent to showing
			\[ \lim_{k\to\infty} T(\pullback{\phi_k}\omega) = T(\pullback{\psi}\omega), \quad \text{or} \quad \lim_{k\to\infty} T(\pullback{\phi_k}\omega - \pullback{\psi}\omega) = 0. \]
			A current $T$ is of order 0 if it is a continuous linear functional on the space of $p$-forms equipped with the $C^0$ (supremum) norm. This means there exists a constant $C_T$ such that for any form $\eta$,
			\[ |T(\eta)| \le C_T \norm{\eta}_0. \]
			Applying this to our difference form $\eta_k = \pullback{\phi_k}\omega - \pullback{\psi}\omega$:
			\[ |T(\pullback{\phi_k}\omega - \pullback{\psi}\omega)| \le C_T \cdot \norm{\pullback{\phi_k}\omega - \pullback{\psi}\omega}_0. \]
			From our hypothesis and Part 1, we know that the right-hand side converges to zero as $k \to \infty$. Therefore, the left-hand side must also converge to zero. This proves the weak-* convergence. For general currents $T$ of arbitrary order, continuity requires convergence of the test form in the $C^\infty$ topology. The hypothesis only establishes $C^0$ convergence of the pullback $\pullback{\phi_k}\omega$. Therefore, the weak-* convergence cannot be concluded for all currents $T \in \Dp{p}$ based solely on this hypothesis. However, for order 0 currents, the result holds.
		\end{itemize}
	\end{proof}

	\begin{corollary}[Uniform Convergence of Pulled-Back Metric]\label{cor:metric_convergence}
		Let $M$ be a compact smooth manifold with a Riemannian metric $g$. Let $\{\phi_k\}_{k=1}^\infty \subseteq \Diff(M)$ be a sequence of smooth diffeomorphisms with $\phi_k \xrightarrow{C^0} \psi \in \Diff(M)$. Assume that for any closed 1-form $\eta$, $\phi_k^*\eta \to \psi^*\eta$ uniformly. Then the sequence of pulled-back metric tensors $\{\phi_k^* g\}$ converges uniformly to $\psi^* g$. That is, $\lim_{k\to\infty} \abs{\phi_k^*g - \psi^*g}_0 = 0,$ 
		where $\abs{\cdot}_0$ denotes the supremum norm for $(0,2)$-tensors, defined as
		$ \abs{T}_0 = \sup_{x \in M} \sup_{\substack{v,w \in T_xM \\ \norm{v}_g=1, \norm{w}_g=1}} \abs{T_x(v,w)}.$
	\end{corollary}
	
	\begin{proof}
		Let $\{\mathcal{V}_j, x_j=(x_j^1, \dots, x_j^n)\}_{j=1}^N$ be a finite cover by coordinate charts and $\{\tau_j\}_{j=1}^N$ a subordinate partition of unity. Write $g = \sum_{j=1}^N \tau_j g =: \sum_{j=1}^N g^j$, where $g^j$ is a smooth $(0,2)$-tensor supported in $\mathcal{V}_j$. By linearity of pullback and the triangle inequality for the supremum norm, it suffices to show uniform convergence $\phi_k^*g^j \to \psi^*g^j$ for each $j$. Fix $j$ and drop the superscript. Let $g$ now denote $g^j$, supported in the chart $(\mathcal{V}, x=(x^1, \dots, x^n))$. In local coordinates, $g = \sum_{a,b=1}^n g_{ab}(x) dx^a \otimes dx^b$, where $g_{ab} \in C^\infty_c(\mathcal{V})$. We have the pullback formula:
		\[ \phi_k^*g = \sum_{a,b} (g_{ab} \circ \phi_k) \phi_k^*(dx^a) \otimes \phi_k^*(dx^b), \quad \psi^*g = \sum_{a,b} (g_{ab} \circ \psi) \psi^*(dx^a) \otimes \psi^*(dx^b). \]
		This leads to the decomposition:
		\begin{align*} \phi_k^*g - \psi^*g ={}& \sum_{a,b} (g_{ab} \circ \phi_k - g_{ab} \circ \psi) \phi_k^*(dx^a) \otimes \phi_k^*(dx^b) \quad (\text{Term A}) \\ &+ \sum_{a,b} (g_{ab} \circ \psi) \left[ \phi_k^*(dx^a) \otimes \phi_k^*(dx^b) - \psi^*(dx^a) \otimes \psi^*(dx^b) \right] \quad (\text{Term B}). \end{align*}
		As argued in the proof of Lemma \ref{Key-0-current}, the assumption of uniform convergence for pullbacks of closed 1-forms implies that $\phi_k^*(dx^a) \to \psi^*(dx^a)$ uniformly on the compact set $K = \supp(g)$ (if necessary use a suitable bump function to globalize the locally defined forms). Let $\omega_k^a = \phi_k^*(dx^a)|_K$ and $\omega^a = \psi^*(dx^a)|_K$. Then $\omega_k^a \to \omega^a$ uniformly on $K$, and the sequences $\{\omega_k^a\}_k$ are uniformly bounded on $K$.	For Term A on $K$, the coefficients $(g_{ab} \circ \phi_k - g_{ab} \circ \psi) \to 0$ uniformly. The tensor product term $\omega_k^a \otimes \omega_k^b$ is uniformly bounded. Therefore, Term A converges to zero uniformly on $K$. For Term B on $K$, the coefficients $(g_{ab} \circ \psi)$ are bounded. We analyze the term in brackets:
		\begin{align*}
			\omega_k^a \otimes \omega_k^b - \omega^a \otimes \omega^b &= (\omega_k^a - \omega^a) \otimes \omega_k^b + \omega^a \otimes (\omega_k^b - \omega^b).
		\end{align*}
		Since $\omega_k^a - \omega^a \to 0$ uniformly, $\omega_k^b$ is bounded, $\omega^a$ is bounded, and $\omega_k^b - \omega^b \to 0$ uniformly (all on $K$), the norm of this bracketed term converges to zero uniformly on $K$. Thus, Term B converges to zero uniformly on $K$. Since $g = g^j$ is supported in $K$, the difference $\phi_k^*g - \psi^*g$ is zero outside $K$. The supremum norm over $M$ is achieved on $K$. Thus, 
		\[ \abs{\phi_k^*g^j - \psi^*g^j}_0 = \sup_K \abs{\phi_k^*g^j - \psi^*g^j} \le \sup_K \abs{\text{Term A}} + \sup_K \abs{\text{Term B}} \to 0. \]
		Since this holds for each $j$ in the finite sum $g = \sum_{j=1}^N g^j$, we conclude that $\abs{\phi_k^*g - \psi^*g}_0 \to 0$ uniformly on $M$.
	\end{proof}

	\subsection{Proof of Theorem \ref{TC}}

	\begin{proof}[Proof of Theorem \ref{TC} via Polyhedral Approximation]
		The proof proceeds in four steps.
		
		\textbf{Step 1: Reduction to Polyhedral Chains.}
		Let $T \in \Rectifiables{k}(M)$. By the Polyhedral Approximation Theorem (see Federer \cite{Federer69}), for any $\epsilon > 0$, there exists a polyhedral k-chain $P$ such that $\norm{T - P}_F < \epsilon$. A polyhedral chain is a finite formal sum of singular simplices, $P = \sum c_j [\sigma_j]$. We use the triangle inequality to estimate the flat norm of the difference:
		\begin{align*}
			\norm{\psi_{i*}T - \psi_*T}_F &\le \norm{\psi_{i*}T - \psi_{i*}P}_F + \norm{\psi_{i*}P - \psi_*P}_F + \norm{\psi_*P - \psi_*T}_F \\
			&= \norm{\psi_{i*}(T - P)}_F + \norm{\psi_{i*}P - \psi_*P}_F + \norm{\psi_*(T - P)}_F.
		\end{align*}
		The pushforward map $\phi_*$ is continuous with respect to the flat norm for any fixed diffeomorphism $\phi$. Specifically, $\norm{\phi_*A}_F \le C_\phi \norm{A}_F$, where $C_\phi$ depends on the Lipschitz constant of $\phi$. Since $\psi_i \xrightarrow{C^0} \psi$ and all maps are smooth on a compact manifold, the family of Lipschitz constants $\{Lip(\psi_i), Lip(\psi)\}$ is uniformly bounded for large $i$. Let this bound be $L$. Thus, the first and third terms are bounded by $L\norm{T - P}_F < L\epsilon$. The problem is now reduced to showing that for any fixed polyhedral chain $P$,
		\[ \lim_{i\to\infty} \norm{\psi_{i*}P - \psi_*P}_F = 0. \]
		
		\textbf{Step 2: Reduction to a Single Simplex.}
		By linearity of the pushforward and the flat norm's triangle inequality, it suffices to prove the result for a single k-simplex, which we denote $T_\sigma = [\sigma]$, where $\sigma: \Delta^k \to M$ is a smooth map from the standard k-simplex $\Delta^k \subset \mathbb R^{k+1}$. We need to show:
		\[ \lim_{i\to\infty} \norm{\psi_{i*}[\sigma] - \psi_*[\sigma]}_F = \lim_{i\to\infty} \norm{[\psi_i \circ \sigma] - [\psi \circ \sigma]}_F = 0. \]
		
		\textbf{Step 3: The Combinatorial Homotopy (Prism Construction).}
		Let $\sigma_i = \psi_i \circ \sigma$ and $\sigma_\psi = \psi \circ \sigma$. We construct a singular $(k+1)$-chain $H_i$ that serves as a homotopy between $\sigma_i$ and $\sigma_\psi$. This chain is a prism. Define $H_i: [0,1] \times \Delta^k \to M$ by
		\[ H_i(t, u) = \exp_{\sigma_\psi(u)}\left( t \cdot \exp_{\sigma_\psi(u)}^{-1}(\sigma_i(u)) \right),\]
		for large $i$. This is a geodesic homotopy connecting each point on the image of $\sigma_\psi$ to the corresponding point on the image of $\sigma_i$. Since $\psi_i \xrightarrow{C^0} \psi$, the distance between $\sigma_i(u)$ and $\sigma_\psi(u)$ goes to zero uniformly in $u$, so this homotopy is well-defined for large $i$. The boundary of this prism chain $[H_i]$ is given by the standard formula from singular homology:
		\[ \partial [H_i] = [\sigma_i] - [\sigma_\psi] - [\text{sides}], \]
		where `sides` denotes the chain formed by the homotopy on the boundary of $\Delta^k$. Rearranging, we get a decomposition for the flat norm: $[\sigma_i] - [\sigma_\psi] = \partial [H_i] + [\text{sides}].$  The decomposition is $T=R+\partial S$. Here, $R=[\text{sides}]$ and $S=[H_i]$. So,
		\[ \norm{[\sigma_i] - [\sigma_\psi]}_F \le \norm{[\text{sides}]}_M + \norm{[H_i]}_M. \]

	\textbf{Step 4: Mass Estimation and Use of the Main Hypothesis.}
	We must show that the mass of the prism, $\norm{[H_i]}_M$, and the mass of its boundary sides, $\norm{[\text{sides}]}_M$, both converge to zero as $i \to \infty$. The mass of the singular $(k+1)$-chain $[H_i]$ is given by the integral of its $(k+1)$-dimensional Jacobian over its domain:
	\[ \norm{[H_i]}_M = \int_{[0,1] \times \Delta^k} J_{H_i}(t,u) \, dt \, du, \]
	where $J_{H_i}$ is the Jacobian of the map $H_i$ with respect to the metric $g$ on $M$. The geometry of the prism map $H_i$ is that of a thin shell. Its "height" at any point $u \in \Delta^k$ is the geodesic distance $L_i(u) = d_g(\sigma_\psi(u), \sigma_i(u))$. Due to the $C^0$-convergence of $\psi_i$ to $\psi$, the maximum height converges to zero:
	\[ h_i := \sup_{u \in \Delta^k} L_i(u) \le \sup_{x \in \im(\sigma)} d_g(\psi(x), \psi_i(x)) \le d_{C^0}(\psi_i, \psi) \to 0. \]
	The volume of this thin prism can be bounded by its maximum height multiplied by a term related to the $k$-dimensional area of its cross-sections. To control this area term, we must control the derivatives of $\psi_i$, which requires the main hypothesis of the theorem. This is where the main hypothesis of the theorem is crucial. The assumption is that for any closed 1-form $\beta$, the sequence $\pullback{\psi_i}\beta$ converges uniformly.
	\begin{enumerate}
		\item By Corollary \ref{cor:metric_convergence}, this assumption implies that the sequence of pulled-back metric tensors $\{\psi_i^*g\}$ converges uniformly to $\psi^*g$.
		
		\item The Jacobian of the map $\sigma_i = \psi_i \circ \sigma$ at a point $u \in \Delta^k$ is computed using the metric pulled back to $\Delta^k$, which is $(\sigma_i)^*g = \sigma^*(\psi_i^*g)$. Similarly, the Jacobian of $\sigma_\psi$ depends on $\sigma^*(\psi^*g)$.
		
		\item Since $\sigma$ is a fixed smooth map on the compact domain $\Delta^k$, the uniform convergence $\psi_i^*g \to \psi^*g$ on $M$ implies the uniform convergence of the pulled-back tensors $\sigma^*(\psi_i^*g) \to \sigma^*(\psi^*g)$ on $\Delta^k$.
		
		\item The Jacobian, $J_k(\sigma_i)$, is a continuous function of the metric tensor (locally, it is the square root of the determinant of the metric components). Thus, the Jacobians converge uniformly: $J_k(\sigma_i) \to J_k(\sigma_\psi)$ for $u \in \Delta^k$.
		
		\item A uniformly convergent sequence of continuous functions on a compact set is uniformly bounded. Therefore, there exists a constant $C_\sigma > 0$ such that $J_k(\sigma_i(u)) \le C_\sigma$ for all $i$ and all $u \in \Delta^k$.
	\end{enumerate}
	
	The map $H_i(t,u)$ smoothly interpolates between $\sigma_\psi$ and $\sigma_i$. Its differential $d_u H_i$ (the differential in the $\Delta^k$ directions) will also be uniformly bounded. This means the $k$-dimensional area of any slice $H_i(t, \cdot)$ is uniformly bounded by some constant $A_{\max}$.	The $(k+1)$-dimensional volume element of the prism at $(t,u)$ is bounded by the product of the length element in the $t$ direction and the $k$-dimensional area element of the slice. The length of the velocity vector $\frac{\partial H_i}{\partial t}$ is precisely $L_i(u) \le h_i$. A rigorous bound on the Jacobian gives:
	\[ J_{H_i}(t,u) \le L_i(u) \cdot (\text{Jacobian of the slice map } u \mapsto H_i(t,u)) \le h_i \cdot C' \]
	for some uniform constant $C'$. Therefore, we can bound the mass:
	\[ \norm{[H_i]}_M = \int_{\Delta^k} \int_0^1 J_{H_i}(t,u) \, dt \, du \le \int_{\Delta^k} \int_0^1 (h_i \cdot C') \, dt \, du = h_i \cdot C' \cdot \Vol(\Delta^k). \]
	Since $h_i \to 0$, we conclude that $\lim_{i\to\infty} \norm{[H_i]}_M = 0$. The argument for the sides is identical. The chain $[\text{sides}]$ is a prism built over the $(k-1)$-dimensional boundary $\partial\sigma$. Its mass is a $k$-dimensional volume. The "height" of this side-prism is still bounded by $h_i$, and its "base" is the $(k-1)$-dimensional manifold $\psi(\partial(\im\sigma))$, which has finite, bounded area. By the same reasoning, its mass is also of order $O(h_i)$, so $\lim_{i\to\infty} \norm{[\text{sides}]}_M = 0$. Since both mass terms on the right-hand side of the flat norm estimate vanish as $i \to \infty$, we have shown $\norm{[\psi_i \circ \sigma] - [\psi \circ \sigma]}_F \to 0$. By the reduction arguments of the previous steps, this completes the proof.
	\end{proof}
	
	\begin{lemma}[Uniform Boundedness of Differentials]
		\label{lem:bounded_diff}
		Let $(M,g)$ be a compact Riemannian manifold. Let $\{\psi_i\}_{i=1}^\infty \subseteq \Diff(M)$ be a sequence of diffeomorphisms converging in $C^0$ to a diffeomorphism $\psi$. If the sequence of pulled-back metric tensors $\{\psi_i^*g\}$ converges uniformly to $\psi^*g$, then the sequence of operator norms of the differentials, $\{\norm{d\psi_i}_{C^0}\}_{i=1}^\infty$, is uniformly bounded.
	\end{lemma}
	
	\begin{proof}
		We proceed by contradiction. The operator norm of the differential of $\psi_i$ is defined as
		\[ \norm{d\psi_i}_{C^0} = \sup_{x \in M} \norm{(d\psi_i)_x}_{op} = \sup_{x \in M} \left( \sup_{v \in T_xM, \norm{v}_g=1} \norm{(d\psi_i)_x v}_g \right). \]
		By the definition of the pullback metric, $(\psi_i^*g)_x(v,v) = g_{\psi_i(x)}((d\psi_i)_x v, (d\psi_i)_x v) = \norm{(d\psi_i)_x v}_g^2$. This allows us to express the operator norm of the differential in terms of the pullback metric:
		\[ \norm{d\psi_i}_{C^0} = \sup_{x \in M} \left( \sup_{v \in T_xM, \norm{v}_g=1} \sqrt{(\psi_i^*g)_x(v,v)} \right). \]
	Assume, for the sake of contradiction, that the sequence of norms $\{\norm{d\psi_i}_{C^0}\}$ is not uniformly bounded. This implies the existence of a subsequence (which we re-index by $i$) such that
	$ \lim_{i\to\infty} \norm{d\psi_i}_{C^0} = \infty.$ 	The norm $\norm{d\psi_i}_{C^0}$ is the maximum value of the continuous function $(x,v) \mapsto \norm{(d\psi_i)_x v}_g$ on the compact unit tangent bundle $SM$. By the Extreme Value Theorem, this maximum is attained. Thus, for each $i$ in this subsequence, we can find a point $x_i \in M$ and a unit vector $v_i \in T_{x_i}M$ where this maximum is achieved: $\norm{(d\psi_i)_{x_i} v_i}_g = \norm{d\psi_i}_{C^0}.$ 
	Squaring this and using the identity $(\psi_i^*g)_x(v,v) = \norm{(d\psi_i)_x v}_g^2$, we get
	$(\psi_i^*g)_{x_i}(v_i, v_i) = \left(\norm{d\psi_i}_{C^0}\right)^2.$  The pointwise operator norm of the tensor $(\psi_i^*g)_x$ is given by $\norm{(\psi_i^*g)_x}_{op} = \sup_{\norm{v}_g=1} (\psi_i^*g)_x(v,v)$, since $\psi_i^*g$ is a positive-definite symmetric tensor. Therefore,
		\[ \norm{(\psi_i^*g)_{x_i}}_{op} \ge (\psi_i^*g)_{x_i}(v_i, v_i) \ge \frac{1}{4} \left(\norm{d\psi_i}_{C^0}\right)^2. \]
		The $C^0$-norm of the tensor $\psi_i^*g$ is the supremum of its pointwise operator norms over $M$:
		\[ \norm{\psi_i^*g}_{C^0} = \sup_{x \in M} \norm{(\psi_i^*g)_x}_{op} \ge \norm{(\psi_i^*g)_{x_i}}_{op}. \]
		Combining the inequalities, we find $\norm{\psi_i^*g}_{C^0} \ge \frac{1}{4} \left(\norm{d\psi_i}_{C^0}\right)^2.$ 
		Since we assumed that $\norm{d\psi_i}_{C^0} \to \infty$ for our subsequence, this implies that the sequence of tensor norms $\{\norm{\psi_i^*g}_{C^0}\}$ must also diverge to infinity. However, the premise of the lemma is that the sequence of tensors $\{\psi_i^*g\}$ converges uniformly to $\psi^*g$. A fundamental result of analysis is that any convergent sequence in a normed space is bounded. Thus, there must exist a constant $C > 0$ such that $\norm{\psi_i^*g}_{C^0} \le C$ for all $i$. This is a direct contradiction to our finding that the norms must be unbounded. The assumption that $\{\norm{d\psi_i}_{C^0}\}$ is unbounded must be false. Therefore, the sequence is uniformly bounded.
	\end{proof}

	\begin{proof}[Proof of Theorem \ref{TC} via Functional Duality]
		The proof proceeds in four steps. It relies on the main assumption in a more subtle, but equally critical, manner than the first proof.
		
		\textbf{Step 1: The Dual Characterization of the Flat Norm.}
		The flat norm of a current $A \in \Currents{k}(M)$ has a dual characterization, established by Federer and Fleming. For rectifiable currents, it is given by:
		\[ \norm{A}_F = \sup \left\{ A(\omega) : \omega \in \Forms{k}(M), \norm{\omega}_{C^0} \le 1, \text{ and } \norm{d\omega}_{C^0} \le 1 \right\}. \]
		Let $\mathcal{K} = \{ \omega \in \Forms{k}(M) : \norm{\omega}_{C^0} \le 1, \norm{d\omega}_{C^0} \le 1 \}$ be this set of test forms. To prove the theorem, we must show that the norm of the difference current converges to zero:
		\[ \lim_{i\to\infty} \norm{\psi_{i*}T - \psi_*T}_F = \lim_{i\to\infty} \sup_{\omega \in \mathcal{K}} \left| (\psi_{i*}T - \psi_*T)(\omega) \right| = 0. \]
		
		\textbf{Step 2: Recasting the Problem via Duality.}
		Using the definition of the pushforward, $(\phi_*T)(\omega) = T(\phi^*\omega)$, our goal becomes showing:
		\[ \lim_{i\to\infty} \sup_{\omega \in \mathcal{K}} \left| T(\psi_i^*\omega) - T(\psi^*\omega) \right| = \lim_{i\to\infty} \sup_{\omega \in \mathcal{K}} \left| T(\psi_i^*\omega - \psi^*\omega) \right| = 0. \]
		Since $T$ is a rectifiable current, it is a current of order 0. This means its action on forms is continuous with respect to the $C^0$-norm. That is, there exists a constant equal to the mass of the current, $\norm{T}_M$, such that $|T(\alpha)| \le \norm{T}_M \norm{\alpha}_{C^0}$ for any continuous form $\alpha$. Therefore,
		\[ \left| T(\psi_i^*\omega - \psi^*\omega) \right| \le \norm{T}_M \norm{\psi_i^*\omega - \psi^*\omega}_{C^0}. \]
		The proof now reduces to establishing the following key analytical statement: the pullback $\psi_i^*\omega$ converges to $\psi^*\omega$ in the $C^0$-norm, and this convergence is \emph{uniform} over the entire set of test forms $\mathcal{K}$. That is, we must prove:
		\[ \lim_{i\to\infty} \sup_{\omega \in \mathcal{K}} \norm{\psi_i^*\omega - \psi^*\omega}_{C^0} = 0. \]
		
		\textbf{Step 3: Precompactness of the Test Forms.}
		The set of test forms $\mathcal{K}$ is precompact in the space of continuous $k$-forms equipped with the $C^0$-norm. We verify this using the Arzela-Ascoli theorem on the compact manifold $M$.
		\begin{enumerate}
			\item \textbf{Uniform Boundedness:} By definition, for any $\omega \in \mathcal{K}$, we have $\norm{\omega}_{C^0} = \sup_{x \in M} |\omega_x| \le 1$. The set $\mathcal{K}$ is uniformly bounded.
			\item \textbf{Equicontinuity:} The condition $\norm{d\omega}_{C^0} \le 1$ provides a uniform bound on the first derivatives of the components of $\omega$ in any local coordinate chart. For any two points $x, y \in M$ connected by a minimizing geodesic $\gamma$, the Mean Value Theorem implies $|\omega_x - \omega_y| \le (\sup_M | \nabla \omega |) \cdot d_g(x,y)$. The norm of the covariant derivative $|\nabla \omega|$ is controlled by the norm of $d\omega$. Thus, the condition $\norm{d\omega}_{C^0} \le 1$ ensures that the family of component functions of forms in $\mathcal{K}$ is equicontinuous.
		\end{enumerate}
		By the Arzela-Ascoli theorem, the closure of $\mathcal{K}$ in the $C^0$ topology is compact.

	\textbf{Step 4: The Uniform Convergence Argument and Use of the Main Hypothesis.}
	We now prove the uniform convergence by contradiction. Assume the convergence is not uniform over $\mathcal{K}$. Then there exists an $\epsilon > 0$, a subsequence (which we re-index by $i$), and a sequence of forms $\{\omega_i\}_{i=1}^\infty \subset \mathcal{K}$ such that for all $i$:
	\begin{equation} \label{eq:contradiction_proof2}
		\norm{\psi_i^*\omega_i - \psi^*\omega_i}_{C^0} \ge \epsilon.
	\end{equation}
	Since $\{\omega_i\} \subset \mathcal{K}$ and $\mathcal{K}$ is precompact, there is a subsequence, which we still denote by $\{\omega_i\}$, that converges in the $C^0$-norm to a continuous $k$-form $\omega_\infty$.	We analyze the term on the left of \eqref{eq:contradiction_proof2} using the triangle inequality:
	\begin{align*}
		\norm{\psi_i^*\omega_i - \psi^*\omega_i}_{C^0} &\le \norm{\psi_i^*(\omega_i - \omega_\infty)}_{C^0} + \norm{\psi_i^*\omega_\infty - \psi^*\omega_\infty}_{C^0} + \norm{\psi^*(\omega_i - \omega_\infty)}_{C^0}. \quad (*)\end{align*}
	
	To show that the right-hand side converges to zero, we must control the derivatives of $\psi_i$. The $C^0$-convergence of maps is insufficient. Here we must use the full power of the main hypothesis and its consequences. The theorem assumes that for any closed 1-form $\beta$, $\pullback{\psi_i}\beta$ converges uniformly. This single assumption provides two crucial pieces of information:
	
	\begin{itemize}
		\item \textbf{(A) Uniform Boundedness of Derivatives:} By Corollary \ref{cor:metric_convergence}, the hypothesis implies the uniform convergence of the pulled-back metric tensors: $\psi_i^*g \xrightarrow{C^0} \psi^*g$. As a sequence of continuous functions on a compact set, $\{\norm{\psi_i^*g}_{C^0}\}$ is bounded. The norm of the differential $(d\psi_i)_x$ is related by $(\psi_i^*g)_x(v,v) = g_{\psi_i(x)}((d\psi_i)_x v, (d\psi_i)_x v)$. This relationship ensures that the sequence of operator norms of the differentials, $\{\norm{d\psi_i}_{C^0}\}_{i=1}^\infty$, must be uniformly bounded (Lemma \ref{lem:bounded_diff}). Let $L$ be a constant such that $\norm{d\psi_i}_{C^0} \le L$ for all $i$.
		
		\item \textbf{(B) Uniform Convergence of Derivatives:} The hypothesis itself is stronger than what is needed for the corollary. Applying the hypothesis to a local basis of closed 1-forms in a coordinate chart, such as $\{dx^j\}$, shows that the pullbacks $\psi_i^*(dx^j) = d(\psi_i^j)$ must converge uniformly (If necessary use a bump function for globalization). The components of the 1-form $d(\psi_i^j)$ are precisely the entries of the $j$-th row of the Jacobian matrix of $\psi_i$. Uniform convergence of these forms therefore implies uniform convergence of the components of the Jacobian matrix. Thus, the derivatives converge uniformly: $d\psi_i \to d\psi$ uniformly on $M$.
	\end{itemize}
	
	With these two facts, we can analyze each term in the sum $(*)$:
	\begin{enumerate}
		\item \textbf{First Term:} The pullback operator norm is controlled by the differential. Using Fact (A):
		\[ \norm{\psi_i^*(\omega_i - \omega_\infty)}_{C^0} \le C_k \norm{d\psi_i}_{C^0}^k \norm{\omega_i - \omega_\infty}_{C^0} \le C_k L^k \norm{\omega_i - \omega_\infty}_{C^0}. \]
		Since $\omega_i \to \omega_\infty$ in $C^0$, this term converges to 0.
		
		\item \textbf{Third Term:} The map $\psi$ is a fixed smooth diffeomorphism, so its differential is bounded. The argument is identical:
		\[ \norm{\psi^*(\omega_i - \omega_\infty)}_{C^0} \le C_k \norm{d\psi}_{C^0}^k \norm{\omega_i - \omega_\infty}_{C^0} \to 0. \]
		
		\item \textbf{Second Term:} We need to show $\norm{\psi_i^*\omega_\infty - \psi^*\omega_\infty}_{C^0} \to 0$. The form $\omega_\infty$ is fixed. At any point $x \in M$:
		\[ (\psi_i^*\omega_\infty)_x - (\psi^*\omega_\infty)_x = (\omega_\infty)_{\psi_i(x)} \circ \Lambda^k(d\psi_i)_x - (\omega_\infty)_{\psi(x)} \circ \Lambda^k(d\psi)_x. \]
		This difference goes to zero uniformly for three reasons:
		(i) $\psi_i(x) \to \psi(x)$ uniformly (given);
		(ii) $\omega_\infty$ is uniformly continuous on the compact manifold $M$;
		(iii) $d\psi_i \to d\psi$ uniformly on $M$ (by Fact (B)). This implies $\Lambda^k(d\psi_i)_x \to \Lambda^k(d\psi)_x$ uniformly.
	\end{enumerate}
	All three terms on the right-hand side of $(*)$ converge to 0 as $i \to \infty$. This implies that $\lim_{i\to\infty} \norm{\psi_i^*\omega_i - \psi^*\omega_i}_{C^0} = 0$, which contradicts our assumption \eqref{eq:contradiction_proof2}. Therefore, the initial assumption of non-uniform convergence must be false. The convergence $\norm{\psi_i^*\omega - \psi^*\omega}_{C^0} \to 0$ is indeed uniform over the set $\mathcal{K}$, completing the proof.
	\end{proof}

	\begin{remark}
		The compactness assumption in our results is crucial. We explore this with a sequence of examples.
		
		\begin{itemize}
			\item \textbf{Compact Manifold ($C^0$ vs $C^1$ convergence):}
			Let $M = [0, 1]$ be the unit interval. Define a sequence of smooth maps $\phi_k(x) = x + \frac{1}{k} \sin(kx)$. As $k \to \infty$, $\phi_k(x)$ converges uniformly to the identity map $\psi(x) = x$, since
			$
			\lim_{k \to \infty} \sup_{x \in [0,1]} |\phi_k(x) - \psi(x)| = \lim_{k \to \infty} \sup_{x \in [0,1]} \left|\frac{1}{k} \sin(kx)\right| = 0.
			$
			However, the derivatives do not converge uniformly. We have $\phi_k'(x) = 1 + \cos(kx)$ and $\psi'(x) = 1$, so
			$$
			\sup_{x \in [0,1]} |\phi_k'(x) - \psi'(x)| = \sup_{x \in [0,1]} |\cos(kx)| = 1 \nrightarrow 0.
			$$ 
			This shows that $C^0$-convergence does not imply $C^1$-convergence. (Note that for $k \ge 1$, $\phi_k$ is not a diffeomorphism of $[0,1]$ to itself).
			
			\item \textbf{Uniform Convergence of Pullbacks (Lemma \ref{Key-current}):}
			Let us consider the same sequence $\phi_k(x) = x + \frac{1}{k}\sin(kx)$ and the 1-form $\alpha = dx$. The pullback is
			$
			\phi_k^* \alpha = d(\phi_k(x)) = \phi_k'(x)dx = (1 + \cos(kx))dx.
			$
			This sequence of forms does not converge uniformly to $\psi^*\alpha = dx$. This highlights why the uniform convergence of pullbacks is a strong assumption in Lemma \ref{Key-current}, not a consequence of $C^0$-convergence of the maps alone. The lemma states that if this uniform convergence happens, the limit must be $\psi^*\alpha$.
			
			\item \textbf{Weak Convergence of Measures:}
			Let $M = [0, 1]$ with a measure $\mu$ given by the density $f(x) = 2x$ with respect to Lebesgue measure. Consider the same maps $\phi_k(x) = x + \frac{1}{k} \sin(kx)$ and a test function $g(x) = x$. The pushforward measure $(\phi_k)_*\mu$ acting on $g$ is $\int_0^1 g(\phi_k(x)) d\mu(x)$. We evaluate the integral:
			$$
			I_k = \int_0^1 g(\phi_k(x)) f(x) \, dx = \int_0^1 \left(x + \frac{1}{k} \sin(kx)\right) (2x) \, dx = 2 \left[ \int_0^1 x^2 \, dx + \frac{1}{k} \int_0^1 x \sin(kx) \, dx \right].
			$$
			Using integration by parts, $\int_0^1 x \sin(kx) \, dx = \frac{\sin(k) - k\cos(k)}{k^2}$. Thus,
			$$
			I_k = 2\left[\frac{1}{3} + \frac{1}{k} \left(\frac{\sin(k) - k \cos(k)}{k^2}\right) \right] = \frac{2}{3} + \frac{2(\sin(k) - k \cos(k))}{k^3}.
			$$
			As $k \to \infty$, since $|\sin(k)| \leq 1$ and $|\cos(k)| \leq 1$, the second term vanishes:
			$
			\lim_{k \to \infty} I_k = \frac{2}{3}.
			$
			This matches the integral with respect to the limit map $\psi(x)=x$:
			$
			\int_0^1 g(\psi(x)) f(x) \, dx = \int_0^1 x(2x) \, dx = \frac{2}{3}.
			$
			This demonstrates the weak-* convergence of the pushforward measures for 0-forms (functions), which is a standard result.
			
			\item \textbf{Non-compactness:} Let $M = \mathbb{R}$. Consider the sequence $\phi_k(x) = x + \frac{1}{k} x^2$. For any fixed $x$, $\phi_k(x) \to x$ as $k \to \infty$. However, the convergence is not uniform on $\mathbb{R}$:
			$$
			\sup_{x \in \mathbb{R}}|\phi_k(x)-x| = \sup_{x \in \mathbb{R}} \left|\frac{1}{k}x^2\right| = \infty.
			$$ 
			Let us check the pullback of the 1-form $\alpha = dx$. We have, 
			$
			\phi_k^* \alpha = d\left(x + \frac{1}{k} x^2\right) = \left(1 + \frac{2}{k}x\right) dx.
			$
			The difference $\phi_k^* \alpha - \alpha = \frac{2}{k}x \, dx$ does not converge to zero uniformly on $\mathbb{R}$, as its supremum norm is infinite for any $k$. This shows that the compactness of $M$ is a necessary condition for the uniform convergence results.
			
		\end{itemize}
	\end{remark}
	\begin{lemma}[Conditional $C^0$-Continuity of the Pullback Map]
		\label{lem:C0_continuity_conditional}
		Let $(M, g)$ be a smooth, oriented, compact Riemannian manifold, and let $\alpha \in \Omega^k(M)$ be a fixed smooth $k$-form. Let $\{\phi_k\}_{k=1}^\infty$ be a sequence in $\mathrm{Diff}^\infty(M)$ converging to a diffeomorphism $\psi$ in the $C^0$-metric. If the sequence of differentials $\{d\phi_k\}$ additionally satisfies:
		\begin{enumerate}
			\item[(i)] The operator norms are uniformly bounded: there exists $L>0$ such that $\norm{d\phi_k}_{C^0} \le L$ for all $k$.
			\item[(ii)] The differentials converge pointwise: for every $x \in M$, $\lim_{k\to\infty} (d\phi_k)_x = (d\psi)_x$.
		\end{enumerate}
		Then the sequence of pulled-back forms converges in the $L^2$-norm: $\lim_{k\to\infty} \norm{\phi_k^*\alpha - \psi^*\alpha}_2 = 0$.
	\end{lemma}
	
	\begin{proof}
		We aim to show that $\lim_{k\to\infty} \int_M |\phi_k^*\alpha(x) - \psi^*\alpha(x)|_g^2 \, d\mathrm{Vol}_g = 0$.
		Let us define a sequence of non-negative functions $f_k: M \to \mathbb{R}$ by $f_k(x) = |\phi_k^*\alpha(x) - \psi^*\alpha(x)|_g^2.$ First, we establish the pointwise convergence of this sequence. At any fixed point $x \in M$, the pullback is given by $(\phi_k^*\alpha)_x = (\alpha)_{\phi_k(x)} \circ \Lambda^k((d\phi_k)_x)$. Since $\phi_k \to \psi$ in $C^0$, we have $\phi_k(x) \to \psi(x)$. By assumption (ii), we have $(d\phi_k)_x \to (d\psi)_x$. Because $\alpha$ is a smooth form and the exterior power operation is continuous, we get pointwise convergence of the forms:
		$\lim_{k\to\infty} (\phi_k^*\alpha)_x = (\alpha)_{\psi(x)} \circ \Lambda^k((d\psi)_x) = (\psi^*\alpha)_x.$ 
		Therefore, the sequence of functions $f_k(x)$ converges pointwise to $0$ for all $x \in M$.	Next, to justify interchanging the limit and the integral, we use the Lebesgue Dominated Convergence Theorem. We must find an integrable function $G(x)$ on $M$ such that $f_k(x) \le G(x)$ for all $k$. We have:
	$$
			f_k(x) = |\phi_k^*\alpha(x) - \psi^*\alpha(x)|_g^2 
			\le \left( |\phi_k^*\alpha(x)|_g + |\psi^*\alpha(x)|_g \right)^2.$$ 
		The norm of a pullback form is bounded by the norm of the differential. That is, there exists a constant $C_k'$ (depending on the degree $k$) such that $|\phi_k^*\alpha(x)|_g \le C_k' \norm{\alpha}_{C^0} \norm{d\phi_k}_{C^0}^k$. Using assumption (i), we have a uniform bound independent of $k$: $|\phi_k^*\alpha(x)|_g \le C_k' \norm{\alpha}_{C^0} L^k.$ 
		The term $|\psi^*\alpha(x)|_g$ is also bounded on the compact manifold $M$ by a constant, say $C_\psi = \norm{\psi^*\alpha}_{C^0}$.
		Thus, the functions $f_k(x)$ are uniformly bounded by a single constant: $f_k(x) \le \left( C_k' \norm{\alpha}_{C^0} L^k + C_\psi \right)^2 =: G_{const}.$ 
		The constant function $G(x) = G_{const}$ is integrable on the compact manifold $M$. By the Dominated Convergence Theorem, we can exchange the limit and the integral:
		\[ \lim_{k\to\infty} \norm{\phi_k^*\alpha - \psi^*\alpha}_2^2 = \lim_{k\to\infty} \int_M f_k(x) \, d\mathrm{Vol}_g = \int_M \left(\lim_{k\to\infty} f_k(x)\right) \, d\mathrm{Vol}_g = \int_M 0 \, d\mathrm{Vol}_g = 0. \]
		This completes the proof.
	\end{proof}
	\section{Applications}\label{sec:applications}

	Here are corollaries of our main results.

	\begin{theorem}\label{Key-1-c} ($C^0$-rigidity of Symplectomorphisms)
		Let $(M, \omega)$ be a compact connected symplectic manifold. Then,  $\overline{\Symp(M, \omega)}^{C^0} \cap \Diff(M) = \Symp(M, \omega).$ 
	\end{theorem}
	\begin{proof}
		Let $\psi\in \overline{\Symp(M, \omega)}^{C^0}\cap \Diff(M)$. Then there exists a sequence $\{\psi_i\}\subseteq \Symp(M, \omega)$ such that $\psi_i\xrightarrow{C^0}\psi$. Let $T$ be an arbitrary rectifiable 2-current. Since rectifiable currents are order 0 currents, we can apply Lemma \ref{Key-current}(2) to get weak-* convergence of the pushforwards: $\lim_{i\to\infty} (\psi_i)_\ast T(\omega) = \psi_\ast T(\omega).$ 
		By definition, this is $\lim_{i\to\infty} T(\psi_i^*\omega) = T(\psi^*\omega)$.
		Since each $\psi_i$ is a symplectomorphism, $\psi_i^*\omega = \omega$ for all $i$. So, the limit becomes $T(\omega) = T(\psi^*\omega)$, which implies $T(\omega - \psi^*\omega) = 0$.
		Since this holds for all rectifiable 2-currents $T$, the 2-form $\omega - \psi^*\omega$ must be zero. (If it were non-zero at a point, one could construct a small 2-disc current $T$ for which the action would be non-zero). Thus, $\psi^*\omega = \omega$, and $\psi \in \Symp(M, \omega)$.
	\end{proof}
	
	\begin{remark}
		This result, originally due to Eliashberg \cite{Eliashberg81, Eliashberg87}, shows that the group of smooth symplectomorphisms is closed under $C^0$ limits within the space of smooth diffeomorphisms.
	\end{remark}

	\subsection*{Example in symplectic geometry}
	\label{sec:symplectic_example}
	\noindent
	Consider the standard symplectic torus $(\mathbb T^2, \omega = dx \wedge dy)$. A symplectomorphism is a linear transformation with determinant 1, like $\psi(x, y) = (x + y, y)$. Now consider a sequence of diffeomorphisms that converge to $\psi$, but are not themselves symplectomorphisms. For $i \in \mathbb{N}$, let
	$$
	\psi_i(x,y) = \left(x+y - \frac{1}{i}\sin(ix), y\right).
	$$
	Since $\sup |\frac{1}{i}\sin(ix)| = \frac{1}{i} \to 0$, we have $\psi_i \xrightarrow{C^0} \psi$. Let $T = [\mathbb{T}^2]$ be the current of integration over the torus. We expect $(\psi_i)_\ast T(\omega) \to \psi_\ast T(\omega)$. First, we compute the pullback $\psi_i^\ast\omega$:
	$$	\psi_i^\ast\omega = d\left(x+y - \frac{1}{i}\sin(ix)\right) \wedge dy 
	= \left((1-\cos(ix))dx + dy\right) \wedge dy 
	= (1-\cos(ix))dx \wedge dy.$$
	
	We compute the action of the pushforward current: $
	(\psi_i)_\ast T(\omega) = \int_{\mathbb T^2} \psi_i^\ast\omega = \int_{\mathbb T^2} (1-\cos(ix)) dx\wedge dy.
	$
	Assuming $i$ is an integer, this integral evaluates to
	$$
	\int_{0}^{2\pi}\int_{0}^{2\pi} (1-\cos(ix)) \,dx\,dy = \int_{0}^{2\pi} \left[ x - \frac{\sin(ix)}{i} \right]_0^{2\pi} \,dy = \int_{0}^{2\pi} 2\pi \,dy = 4\pi^2 = \text{Area}(\mathbb T^2).
	$$
	For the limit map, $\psi^*\omega = d(x+y)\wedge dy = dx \wedge dy$, so $
	\psi_\ast T(\omega) = \int_{\mathbb T^2} dx\wedge dy = \text{Area}(\mathbb T^2).
	$
	Thus, we have explicitly shown that $\lim_{i\to\infty} (\psi_i)_\ast T(\omega) = \psi_\ast T(\omega)$, as predicted by the weak-* convergence. This example demonstrates how the integral of a pulled-back form remains constant even when the form itself, $\psi_i^*\omega$, is not constant and only converges to $\psi^*\omega$ in a weak sense (not uniformly).
	\subsubsection{Illustrative Example (A)}
	
	We consider the torus $\mathbb{T}^2$ as $[0, 2\pi] \times [0, 2\pi]$ with opposite sides identified, using coordinates $(\theta, \phi)$. Let $\omega = a(\theta, \phi)d\theta + b(\theta, \phi)d\phi$ be a smooth 1-form, with $a$ and $b$ being smooth, $2\pi$-periodic functions. Let $\alpha > 0$ be a small number such that $e^\alpha \leq 2\pi$ and $\alpha^2 \leq 2\pi$ (e.g., $\alpha=1$). Let the curve $\gamma$ be parameterized by $c(t) = (e^t, t^2)$ for $t \in [0, \alpha]$. The 1-current $T_\gamma$ is: $$
	T_\gamma(\omega) = \int_\gamma \omega = \int_0^\alpha \omega_{c(t)}(c'(t)) dt = \int_0^\alpha [a(e^t, t^2)e^t + b(e^t, t^2)2t] dt,$$ 
	since $c'(t) = (e^t, 2t)$. We use the diffeomorphisms: 
	$\psi(\theta, \phi) = (\theta + \phi, \phi)$ and  
	$\psi_i(\theta, \phi) = (\theta + (1 - \frac{1}{i})\phi, \phi)$, with $\psi_i \xrightarrow{C^0} \psi$.
	The pullbacks are:
	\[
	\psi_i^*\omega = a(\theta + (1 - \frac{1}{i})\phi, \phi) d\theta + \left[ (1 - \frac{1}{i})a(\theta + (1 - \frac{1}{i})\phi, \phi) + b(\theta + (1 - \frac{1}{i})\phi, \phi) \right] d\phi,
	\]
	and $
	\psi^*\omega = a(\theta + \phi, \phi) d\theta + \left[ a(\theta + \phi, \phi) + b(\theta + \phi, \phi) \right] d\phi.$ The pushforward of $T_\gamma$ by $\psi_i$ is $(\psi_{i*}T_\gamma)(\omega) = \int_\gamma \psi_i^*\omega = \int_0^\alpha \psi_i^*\omega_{c(t)}(c'(t)) dt$.
	Explicitly,
	\begin{align*}
		(\psi_{i*}T_\gamma)(\omega) = \int_0^\alpha \left[ a(\dots)e^t + \left( (1 - \tfrac{1}{i})a(\dots) + b(\dots) \right) 2t \right] dt,
	\end{align*}
	where the arguments of $a$ and $b$ are $(e^t + (1 - \frac{1}{i})t^2, t^2)$. Similarly,
	\begin{align*}
		(\psi_*T_\gamma)(\omega) = \int_0^\alpha \left[ a(e^t + t^2, t^2)e^t + \left( a(e^t + t^2, t^2) + b(e^t + t^2, t^2) \right) 2t \right] dt.
	\end{align*}
	
	The integrand for $(\psi_{i*}T_\gamma)(\omega)$ converges uniformly to the integrand for $(\psi_{*}T_\gamma)(\omega)$ on the compact interval $[0, \alpha]$, because $a$ and $b$ are continuous. Therefore, we can interchange the limit and the integral:
	$
	\lim_{i \to \infty} (\psi_{i*}T_\gamma)(\omega) = (\psi_*T_\gamma)(\omega).$ This demonstrates the stability.  $\quad \bigtriangleup
	$

	\begin{theorem}\label{Key-2-c} ($C^0$-rigidity of Cosymplectomorphisms)
		Let $(M, \omega, \eta)$ be a compact connected cosymplectic manifold. Then
		$ \overline{\Cosymp(M, \omega, \eta)}^{C^0}\cap \Diff(M) = \Cosymp(M, \omega, \eta).$ 
	\end{theorem}
	\begin{proof}
		Let $\psi \in \overline{\Cosymp(M, \omega, \eta)}^{C^0} \cap \Diff(M)$. Let $\{\psi_i\} \subseteq \Cosymp(M, \omega, \eta)$ be a sequence with $\psi_i \xrightarrow{C^0} \psi$. Then $\pullback{\psi_i}\omega = \omega$ and $\pullback{\psi_i}\eta = \eta$. An application of the argument in Lemma \ref{Key-current}  implies that $\pullback{\psi}\omega = \omega$ and $\pullback{\psi}\eta = \eta$, so $\psi \in \Cosymp(M, \omega, \eta)$.
	\end{proof}
	\begin{remark} See \cite{Tchuiaga22} for a related result on the identity component of $\Cosymp(M, \omega, \eta)$. \end{remark}

	\begin{theorem}\label{Key-3-c} ($C^0$-rigidity of Contactomorphisms Preserving the Form)
		Let $(M, \alpha)$ be a compact connected contact manifold. Let $\Diff(M, \alpha) = \{\phi \in \Diff(M) : \pullback{\phi}\alpha = \alpha\}$. Then
		$$ \overline{ \Diff(M, \alpha)}^{C^0}\cap \Diff(M) = \Diff(M, \alpha).$$ 
	\end{theorem}
	
	\begin{proof}
		Let $\psi \in \overline{ \Diff(M, \alpha)}^{C^0}\cap \Diff(M)$. There exists a sequence $\{\psi_i\} \subseteq \Diff(M, \alpha)$ such that $\psi_i \xrightarrow{C^0} \psi$. Since $\pullback{\psi_i}\alpha = \alpha$ for all $i$, the sequence of pullbacks $\{\pullback{\psi_i}\alpha\}$ is the constant sequence $(\alpha, \alpha, \dots)$. This sequence converges uniformly to the limit $\theta = \alpha$. By Lemma \ref{Key-current}(1), this limit must be equal to $\pullback{\psi}\alpha$. Therefore, $\pullback{\psi}\alpha = \alpha$, which means $\psi \in \Diff(M, \alpha)$.
	\end{proof}

	\textbf{Illustrative Example (B):} Consider $S^3$ as the unit sphere in $\mathbb{C}^2$. Let $(z_1, z_2)$ be complex coordinates, where $z_j = x_j + iy_j$. The standard contact form on $S^3$ is:
	$$
	\alpha = \frac{i}{2} \sum_{j=1}^2 (z_j d\bar{z}_j - \bar{z}_j dz_j) = x_1 dy_1 - y_1 dx_1 + x_2 dy_2 - y_2 dx_2.
	$$
	Consider a sequence of angles $\theta_i = 1/i \to 0$. Define $\phi_i \in \Diff(S^3)$ as a rotation in the first complex coordinate:
	$
	\phi_i(z_1, z_2) = (e^{i\theta_i} z_1, z_2).
	$
	Each $\phi_i$ preserves the contact form $\alpha$ and is therefore in $\Diff(S^3, \alpha)$. As $i \to \infty$, $\theta_i \to 0$, so $\phi_i$ converges uniformly to the identity map $\psi = \Id$:
	$$
	d_{C^0}(\phi_i, \Id) = \sup_{(z_1,z_2) \in S^3} \|(e^{i\theta_i} z_1, z_2) - (z_1, z_2)\| = \sup |(e^{i\theta_i}-1)z_1| = |e^{i\theta_i}-1| \to 0.
	$$
	The limit map $\psi = \Id$ clearly preserves the contact form, as $\psi^*\alpha = \Id^*\alpha = \alpha$. This confirms the result of Theorem \ref{Key-3-c}: the $C^0$-limit of a sequence of contactomorphisms is itself a contactomorphism. \quad $\bigtriangleup$
	
	\begin{theorem}\label{Key-4-c} ($C^0$-rigidity of Volume-preserving Diffeomorphisms)
		Let $(M, \Omega)$ be a compact connected oriented manifold with volume form $\Omega$. Then
		$ \overline{ \Diff(M, \Omega)}^{C^0}\cap \Diff(M) = \Diff(M, \Omega).$ 
	\end{theorem}
	
	\begin{proof}
		Let $\psi \in \overline{ \Diff(M, \Omega)}^{C^0}\cap \Diff(M)$. There exists a sequence of volume-preserving diffeomorphisms $\{\psi_i\} \subseteq \Diff(M, \Omega)$ such that $\psi_i \xrightarrow{C^0} \psi$. Since $\pullback{\psi_i}\Omega = \Omega$ for all $i$, the sequence of pullbacks converges uniformly to $\Omega$. By Lemma \ref{Key-current}(1), the limit must also be $\pullback{\psi}\Omega$. Therefore, $\pullback{\psi}\Omega = \Omega$.
	\end{proof}

	\begin{theorem}\label{Isom} ($C^0$-rigidity of Isometries)
		Let $(M, g)$ be a smooth, compact, connected Riemannian manifold. Then the group of smooth isometries $\Iso(M,g)$ is $C^0$-closed within the group of smooth diffeomorphisms $\Diff(M)$.
	\end{theorem}
	
	\begin{proof}
		Let $\psi \in \overline{\Iso(M,g)}^{C^0} \cap \Diff(M)$. There is a sequence of isometries $\{\psi_i\}$ such that $\psi_i \xrightarrow{C^0} \psi$. We will use the characterization of isometries via contact geometry on the cotangent bundle.
		A diffeomorphism $\phi: M \to M$ is an isometry if and only if its canonical lift to the unit cotangent bundle, $\Phi: S^*M \to S^*M$, preserves the canonical contact 1-form $\alpha$ on $S^*M$. For each isometry $\psi_i$, its lift $\Psi_i: S^*M \to S^*M$ is a contactomorphism, i.e., $\Psi_i^*\alpha = \alpha$. The $C^0$-convergence $\psi_i \to \psi$ on $M$ implies the $C^0$-convergence of their lifts, $\Psi_i \to \Psi$, on $S^*M$. We now have a sequence of contactomorphisms $\{\Psi_i\}$ converging in $C^0$ to the diffeomorphism $\Psi$.
		By the $C^0$-rigidity of contactomorphisms (Theorem \ref{Key-3-c} applied to the manifold $S^*M$ with form $\alpha$), the limit map $\Psi$ must also be a contactomorphism, meaning $\Psi^*\alpha = \alpha$.
		This implies that its base map $\psi: M \to M$ must be an isometry. Therefore, $\psi \in \Iso(M,g)$.
	\end{proof}
	\section{Illustrations and Further Discussions}
	
	\subsection{Convergence of integrals over submanifolds}
	Assume we have a compact, orientable $k$-submanifold $N \subset M$ and a sequence of diffeomorphisms $\psi_i \xrightarrow{C^0} \psi \in \Diff(M)$ so that for any closed $1$-form $\beta$ the sequence $\pullback{\psi_i} \beta$ converges uniformly in $\mathcal Z^1(M)$. Let $\omega$ be a smooth $k$-form on $M$. Let $[N]$ be the rectifiable current corresponding to integration over $N$. By Theorem \ref{TC}, $(\psi_i)_*[N] \to \psi_*[N]$ in the flat norm, which implies weak-* convergence. Thus, $(\psi_i)_*[N](\omega) \to \psi_*[N](\omega)$. By the definition of the pushforward:
	\[ (\psi_i)_*[N](\omega) = [N](\psi_i^* \omega) = \int_N \psi_i^* \omega = \int_{\psi_i(N)} \omega, \]
	and similarly $\psi_*[N](\omega) = \int_{\psi(N)} \omega$. Therefore, we conclude that
	\[ \lim_{i \to \infty} \int_{\psi_i(N)} \omega = \int_{\psi(N)} \omega. \]
	This shows that not just volumes, but integrals of any smooth $k$-form over the transformed submanifolds, converge. This reflects a preservation of geometric quantities under $C^0$-limits.
	
	\subsection{Homology and linking numbers}
	Let $N_1$ (dimension $k$) and $N_2$ (dimension $n-k-1$) be disjoint, compact, oriented submanifolds of a compact, oriented manifold $M$ (dimension $n$). They represent homology classes and have a well-defined linking number, $\Lk(N_1, N_2)$. Consider a sequence $\psi_i \xrightarrow{C^0} \psi$ so that   for any closed $1$-form $\beta$ the sequence $\pullback{\psi_i} \beta$ converges uniformly in $\mathcal Z^1(M)$. For $i$ large enough, $\psi_i$ is homotopic to $\psi$, so $\psi_i(N_1)$ and $\psi(N_1)$ are homologous, and similarly for $N_2$. This implies that the homology classes are eventually constant. Linking numbers, being homological invariants, are also constant for large $i$: $\Lk(\psi_i(N_1), \psi_i(N_2)) = \Lk(\psi(N_1), \psi(N_2))$. Our theorem provides a geometric underpinning for this topological stability via the convergence of the underlying geometric objects (currents) in the flat norm.
	
	\subsection{Convergence of Minimal Surfaces}
	Let $N$ be a minimal surface (a critical point of the area functional) in $(M,g)$. Consider $\psi_i \xrightarrow{C^0} \psi$ so that   for any closed $1$-form $\beta$ the sequence $\pullback{\psi_i} \beta$ converges uniformly in $\mathcal Z^1(M)$. While the areas of $\psi_i(N)$ converge to the area of $\psi(N)$, it is not immediately obvious that $\psi(N)$ is also minimal. Flat norm convergence controls the "shape" of the surfaces, suggesting that properties related to first variations of area might be preserved. A full proof that $\psi(N)$ is minimal would require additional geometric analysis tools, but flat convergence is a key first step in many such stability arguments.
	
	\subsection{Currents defined by vector fields and flows}
	Let $X$ be a smooth vector field on $M$. We can associate a $1$-current $T_X$ to $X$ by $T_X(\omega) = \int_M \omega(X) \, d\text{vol}_g$ for any 1-form $\omega$. Let $\{\psi_i\}$ be a sequence of diffeomorphisms converging to the identity in $C^0$ so that for a smooth $p$-form $\alpha \in \Omk{p}$ the sequence $\pullback{\phi_k} \alpha$ converges uniformly in $\Omk{p}$. Using Lemma \ref{Key-current}, we have weak-* convergence $(\psi_i)_* T_X \to T_X$. The pushforward action is
	$(\psi_i)_* T_X(\omega) = T_X(\psi_i^*\omega) = \int_M (\psi_i^*\omega)(X) \, d\text{vol}_g.$
	A direct calculation shows this corresponds to a current generated by the pushforward vector field $(\psi_i)_*X$. The weak-* convergence implies that integrals involving these transformed vector fields converge, even if the vector fields $(\psi_i)_*X$ themselves do not converge uniformly.

	\begin{center}
		\section*{Acknowledgments}
		This work is dedicated to unknown Mathematicians. 
		
	\end{center}

	
\end{document}